\documentclass[12pt,a4paper]{amsart}
\usepackage{graphicx} %To include graphics
\usepackage{latexsym}
\usepackage{amssymb}
\usepackage{verbatim}
\usepackage{url}
\usepackage{hyperref}

\newtheorem{thm}{Theorem}[section]
\newtheorem{cor}[thm]{Corollary}
\newtheorem{lem}[thm]{Lemma}
\newtheorem{prop}[thm]{Proposition}
\newtheorem{problem}[thm]{Problem}

\newtheorem{conj}[thm]{Conjecture}

\theoremstyle{definition}

\theoremstyle{remark}
\newtheorem{rmk}{Remark}[section]

\newcommand{\D}{\Delta}
\newcommand{\C}{\mathcal{C}}

\newcommand{\Diag}{\mathcal{D}}
\newcommand{\Q}{\mathcal{Q}}

\title{Constructions of large graphs on surfaces}

\author{Ramiro Feria-Pur\'on}
\email{\texttt{ramiro.feria-puron@uon.edu.au}}
\address{School of
Electrical Engineering and Computer Science,
The University of Newcastle}

\author{Guillermo Pineda-Villavicencio}
\email{\texttt{work@guillermo.com.au}}
\address{Centre for Informatics and Applied Optimisation,
University of Ballarat}
\thanks{Guillermo would like to thank the partial support received by the Australian Research Council Project DP110102011.}


\begin{document}
\begin{abstract}
We consider the degree/diameter problem for graphs embedded in a surface, namely, given a surface $\Sigma$ and integers $\Delta$ and $k$, determine the maximum order $N(\Delta,k,\Sigma)$ of a graph embeddable in $\Sigma$ with maximum degree $\Delta$ and diameter $k$. We introduce a number of constructions which produce many new largest known planar and toroidal graphs. We record all these graphs in the available tables of largest known graphs.

Given a surface $\Sigma$ of Euler genus $g$ and an odd diameter $k$, the  current best asymptotic lower bound for $N(\Delta,k,\Sigma)$ is given by  \[\sqrt{\frac{3}{8}g}\Delta^{\lfloor k/2\rfloor}.\] 
Our constructions produce new graphs of order
\[\begin{cases}6\Delta^{\lfloor k/2\rfloor}& \text{if $\Sigma$ is the Klein bottle}\\  \left(\frac{7}{2}+\sqrt{6g+\frac{1}{4}}\right)\Delta^{\lfloor k/2\rfloor}& \text{otherwise,}\end{cases}\]
thus improving the former value by a factor of $4$.   
\end{abstract}

\maketitle

\section{Introduction}

Given a class $\C$ of graphs and integers $\Delta$ and $k$, the \emph{degree/diameter problem} aims to find the maximum order $N(\Delta,k,\C)$ of a graph in $\C$ with maximum degree $\Delta$ and diameter $k$. For background on this problem the reader is referred to the survey \cite{MS05a}.

Given a surface $\Sigma$, let $\mathcal{G}(\Sigma)$ denote the class of graphs embeddable in $\Sigma$. We set $N(\Delta, k, \Sigma):=N(\Delta, k, \mathcal{G}(\Sigma))$ for simplicity. 

The Moore bound \[1+\Delta+\Delta(\Delta-1)+\ldots+\Delta(\Delta-1)^{k-1}\] provides an upper bound for the order of an arbitrary graph with maximum degree $\Delta$ and diameter $k$. This bound, however, is a very rough upper bound when considering graphs on surfaces. The current best upper bound for $N(\Delta,k,\Sigma)$ was provided by \v{S}iagiov\'a and Simanjuntak \cite{SR04} who showed that, for every  
surface $\Sigma$ of Euler genus $g$, every $k\ge2$ and every $\Delta\ge3$,  \[N(\Delta,k,\Sigma)\le cgk\Delta^{\lfloor k/2\rfloor}.\]

Before continuing we clarify what we mean by a surface and its Euler genus, as different authors adopt different definitions. A {\it surface} is a compact (connected) 2-manifold (without boundary). Every surface is homeomorphic to the sphere with $h$ handles or the sphere with $c$ cross-caps \cite[Theorem~3.1.3]{MohTho01}. The sphere with $h$ handles has {\it Euler genus} $2h$, while the sphere with $c$ cross-caps has {\it Euler genus} $c$.  

 The bound $cgk\Delta^{\lfloor k/2\rfloor}$ still seems to be rough for graphs on surfaces of Euler genus $g$, as demonstrated in \cite{PVW12}.  For the class $\mathcal{P}$ of planar graphs the paper \cite{PVW12} recently showed that, for a fixed $k$, the limit \[\lim_{\Delta\rightarrow \infty}\frac{N(\D,k,\mathcal{P})}{\D^{\lfloor k/2\rfloor}}\] is an absolute constant, independent of $\Delta$ or $k$. 
  
Knor and \v{S}ir\'a\v{n} \cite{KS97} proved that, for every surface $\Sigma$, there exists an integer $\Delta_0$ such that, for all $\Delta\ge\Delta_0$, \[N(\Delta, 2, \Sigma)=N(\Delta, 2, \mathcal{P})=\lfloor\frac{3}{2}\Delta\rfloor+1.\] This result motivated Miller and \v{S}ir\'a\v{n} to pose the following problem \cite[pp.~46]{MS05a}.

\begin{problem}[{\cite[pp.~46]{MS05a}}]
\label{prob:Miller-Siran}
Prove or disprove that, for each surface $\Sigma$ and each diameter $k\ge2$, there exists a constant $\Delta_0$ such that, for maximum degree $\Delta\ge \Delta_0$, $N(\Delta,k,\Sigma)=N(\Delta,k,\mathcal{P})$.
\end{problem}

Problem \ref{prob:Miller-Siran} was answered in the negative by Pineda-Villavicencio and Wood \cite{PVW12}, where the authors proved that, for every  
surface $\Sigma$ of Euler genus $g$, every odd $k\ge3$ and every $\Delta\ge \sqrt{1+24g}+2$, \[N(\Delta,k,\Sigma)\ge\sqrt{\frac{3}{8}g}\Delta^{\lfloor k/2\rfloor}.\]

In Section \ref{sec:LargeGrapSurf} we construct graphs whose orders improve the above lower bound for $N(\Delta,k,\Sigma)$ by a factor of $4$. We  obtain
\[N(\Delta,k,\Sigma)\ge\begin{cases}6\Delta^{\lfloor k/2\rfloor}& \text{if $\Sigma$ is the Klein bottle}\\  \left(\frac{7}{2}+\sqrt{6g+\frac{1}{4}}\right)\Delta^{\lfloor k/2\rfloor}& \text{otherwise.}\end{cases}\] 

Sections \ref{sec:LPlanar} and \ref{sec:LGTorus} are devoted to the construction of new largest known planar and toroidal graphs for maximum degree $3\le\Delta\le 10$ and diameter $2\le k\le10$. In the appendix we provide tables cataloging the largest known such graphs. In the case of planar graphs the existing table was updated with the new orders. For toroidal graphs no such table existed; only a table recording largest regular graphs was available at \cite{Preen_RegToroidal}. We also updated accordingly (created in the case of toroidal graphs \cite{LPP_Toroidal}) the online table of largest known planar graphs \cite{LPP_Planar}.   

Our constructions extend approaches put forward in \cite{FHS98}; Section \ref{sec:MultDiag} explains the methodology.

\section{Multigraphs and diagrams}
\label{sec:MultDiag}
We start from the definition of a \emph{diagram} presented in \cite{FHS98}. A diagram $\Diag$ is a multigraph where edges are labelled  in the form $\alpha(\D, \beta)$ ($\alpha$ and $\beta$ positive integers); see Figure \ref{fig:PodsDiag} (a). An edge in a diagram $\Diag$ is called \emph{thin} if $\alpha=\beta=1$, otherwise it is called \emph{thick}. Similarly, a vertex in $\Diag$ is called \emph{thin} if all its incident edges are thin, otherwise it is called \emph{thick}. For thin edges labels are omitted. The {\it unlabelled degree} of a vertex $v$ in $\Diag$ is the number of edges incident with $v$, while the (labelled)\emph{degree} of a vertex $v$ in $\Diag$ is the sum of all the $\alpha$ values on the labels of the edges incident with $v$. For instance, in Fig.~ \ref{fig:PodsDiag} (a) vertex $v$ has unlabelled degree two and (labelled) degree three. The \emph{weight} of a walk in $\Diag$ is the sum of all the $\beta$ values on the labels of the edges of the walk. An edge $e$ in $\Diag$ with an endvertex of unlabelled degree one is called {\it pending}, otherwise $e$ is called {\it non-pending}. For an integer $k\ge 2$, $\Diag_{\D}^k$ denotes any diagram $\Diag$ with maximum degree $\D$ and labels $\alpha(\D,\beta)$ satisfying $\beta\le k$ for non-pending edges, and $\beta\le \lfloor k/2\rfloor$ for pending edges.

\begin{figure}[!ht]
\begin{center}
\includegraphics[scale=1]{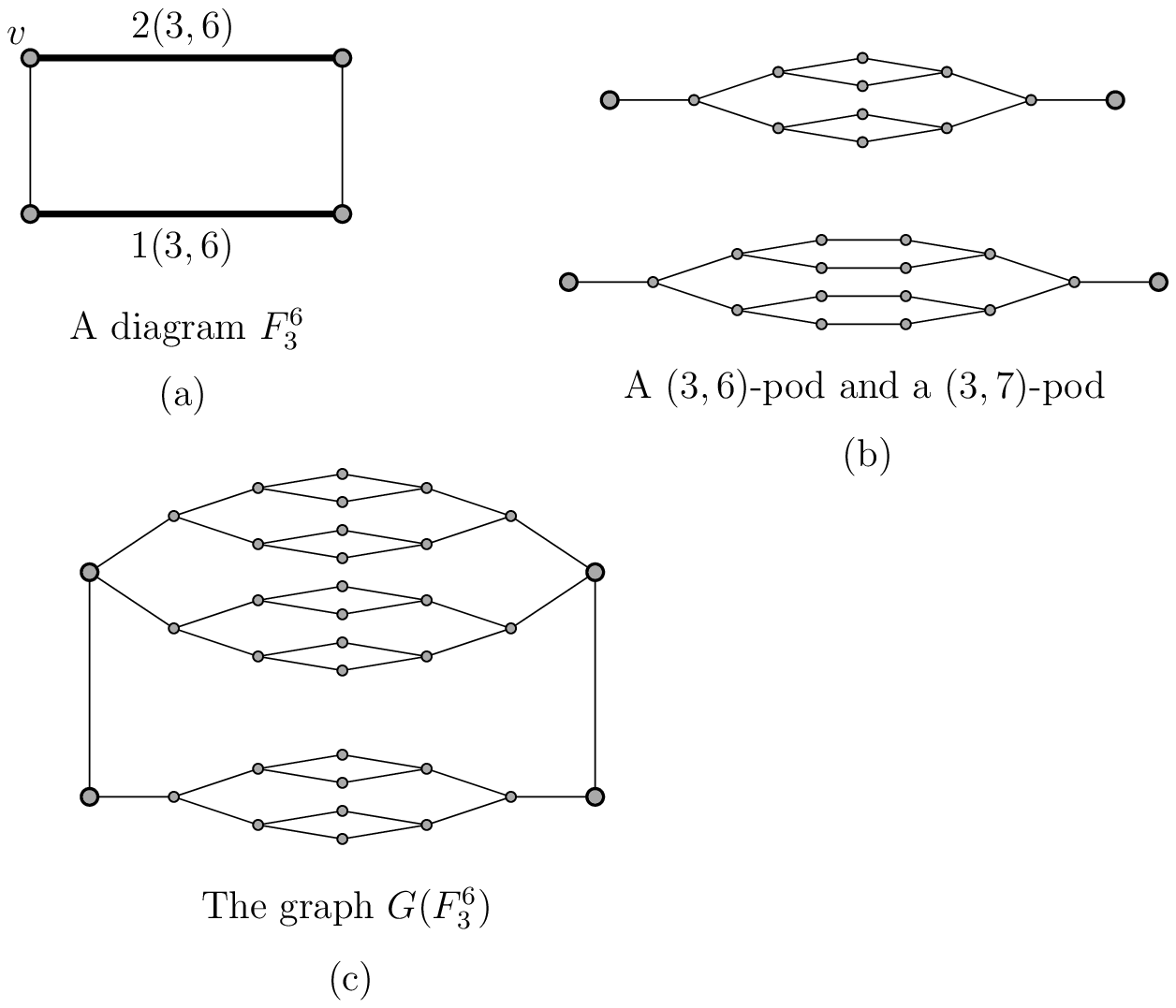}
\caption{}
\label{fig:PodsDiag}
\end{center}
\end{figure}

The {\it depth} of a tree rooted at a vertex $v$ is the length of a longest path from $v$ to the tree leaves. For a positive integer $\gamma$, a \emph{$(\D,\gamma)$-tree} is a tree of depht $\gamma$ with its root and leaves having degree $1$ and all other vertices having degree $\D$. Given a positive integer $\beta$, we call a \emph{$(\D,\beta)$-pod} the planar graph obtained from two $(\D,\lfloor\beta/2\rfloor)$-trees, identifying their leaves if $\beta$ is even and matching their leaves if $\beta$ is odd; see Figure \ref{fig:PodsDiag} (b) for an example. The roots of the two trees used in the pod construction are the \emph{roots} of the pod; the remaining vertices are called \emph{internal}. In a pod, a path linking its roots is called a \emph{vein}. The number of internal vertices in a $(\D, \beta)$-pod with $\beta\ge2$ is $$\frac{\D(\D-1)^{(\beta-2)/2}-2}{\D-2} \quad	\textrm{if $\beta$ is even, }$$and $$\frac{2(\D-1)^{(\beta-1)/2}-2}{\D-2} \quad \textrm{if $\beta$ is odd.}$$

From a diagram $\Diag_{\D}^{k}$ we define a compound graph $G(\Diag_{\D}^{k})$ as follows: a thick non-pending edge $e$ in $\Diag_{\D}^{k}$, labelled by $\alpha(\D, \beta)$, is replaced by $\alpha$ ``disjoint'' $(\D, \beta)$-pods. By ``disjoint'' pods we mean pods that only share their roots. The endvertices of $e$ are identified with the roots of the pods; see Figure \ref{fig:PodsDiag} (c). If instead $e$ is a thick pending edge in $\Diag_{\D}^{k}$, $e$ is replaced by $\alpha$ ``disjoint'' $(\D, \beta)$-trees. The endvertex of $e$ with unlabelled degree other than one is identified with the roots of the trees replacing $e$.

Next we establish relations between $\Diag_{\D}^{k}$ and $G(\Diag_{\D}^{k})$; most of them already appeared  in \cite{FHS98}.

\begin{prop}[{\cite[Lemma 1]{FHS98}}]\label{prop:Degree} The maximum degree of  $G(\Diag_{\D}^{k})$ is at most $\D$.
\end{prop}

\begin{lem}[{\cite[Lemma 2]{FHS98}}]\label{lem:lemma2} Consider a diagram $\Diag_{\D}^{k}$ and the graph $G(\Diag_{\D}^{k})$, and suppose that $e$ and $e'$ are two thick edges of $\Diag_{\D}^{k}$ which lie on a cycle of weight at most $2k+1$. Let $v$ and $v'$ be vertices in $G(\Diag_{\D}^{k})$ of pods corresponding to $e$ and $e'$ respectively. Then the distance in $G(\Diag_{\D}^{k})$ between $v$ and $v'$ is at most $k$.
\end{lem}

\begin{prop}\label{prop:Diameter} Consider a diagram $\Diag_{\D}^{k}$ and the graph $G(\Diag_{\D}^{k})$, and suppose the following conditions hold.
\begin{enumerate}
  \item Any two thick edges of $\Diag_{\D}^{k}$ are contained in a closed walk of weight at most $2k+1$.
  \item For any thin vertex $v$ and any thick edge $e$ of $\Diag_{\D}^{k}$, $v$ and $e$ lie in closed walk of weight at most $2k+1$.
  \item There is  a path of weight at most $k$ between any two thin vertices of  $\Diag_{\D}^{k}$.
\end{enumerate}

Then the graph $G(\Diag_{\D}^{k})$ has diameter at most $k$.
\end{prop}

\begin{proof} If a thick non-pending edge $e$ of $\Diag_{\D}^{k}$ has label $\alpha(\D, \beta)$  then the distance between any two vertices in $G(\Diag_{\D}^{k})$ belonging to the pods which replaced $e$ is at most $k$ \cite[pp.~277]{FHS98}. Since $\beta\le \lfloor k/2\rfloor$, for a thick pending edge $e$ labelled by $\alpha(\D, \beta)$, the distance between any two vertices in $G(\Diag_{\D}^{k})$ belonging to the trees which replaced $e$ is at most $k$ as well. Condition (1) assures that any two vertices of $G(\Diag_{\D}^{k})$ belonging to pods replacing distinct thick edges are at distance at most $k$. Conditions (2) and (3) guarantee that the distance from a thin vertex to any vertex in a pod of $G(\Diag_{\D}^{k})$, and to any other thin vertex, is also at most $k$.
\end{proof}

As we will show in the proofs of Proposition \ref{prop:G(Y)-diameter} and Lemma\ref{lem:BasicLemma}, Condition (1) can often be relaxed to containment in two closed walks of weight $2k+2$, rather than just one closed walk of weight at most $2k+1$.

\section{Large planar graphs with odd diameter}
\label{sec:LPlanar}
Figure \ref{fig:DiagramY} (a) depicts the diagram $C_{\D}^k$ ($\D\ge 4$) suggested in \cite{FHS98}, which gives rise to the largest known planar graphs of maximum degree $\D\ge6$ and odd diameters $k\ge5$. The order of $G(C_{\Delta}^k)$ for odd $k\ge5$ is
\begin{align*}
|G(C_{\Delta}^k)|=(\lfloor\frac{9\D}{2}\rfloor-12)\frac{\D(\D-1)^{\frac{k-3}{2}}-2}{\D-2}+9.
\end{align*}

As noted in \cite{FHS98}, the diagram $C_{\D}^k$ has maximum degree $\D$ and readily satisfies the conditions of Proposition \ref{prop:Diameter}. Thus, the graph $G(C_{\D}^{k})$ has maximum degree $\D$ and diameter $k$. A minor modification of $C_{\Delta}^k$ produces a diagram $Y_{\Delta}^k$ (for odd $k\ge5$) with three additional vertices and, in the case of odd $\D$, with an extra pending edge; see Figures \ref{fig:DiagramY} (b) and (c). Clearly, in both cases $Y_{\Delta}^k$ has maximum degree $\Delta$.

\begin{figure}[!ht]
\begin{center}
\includegraphics[scale=1]{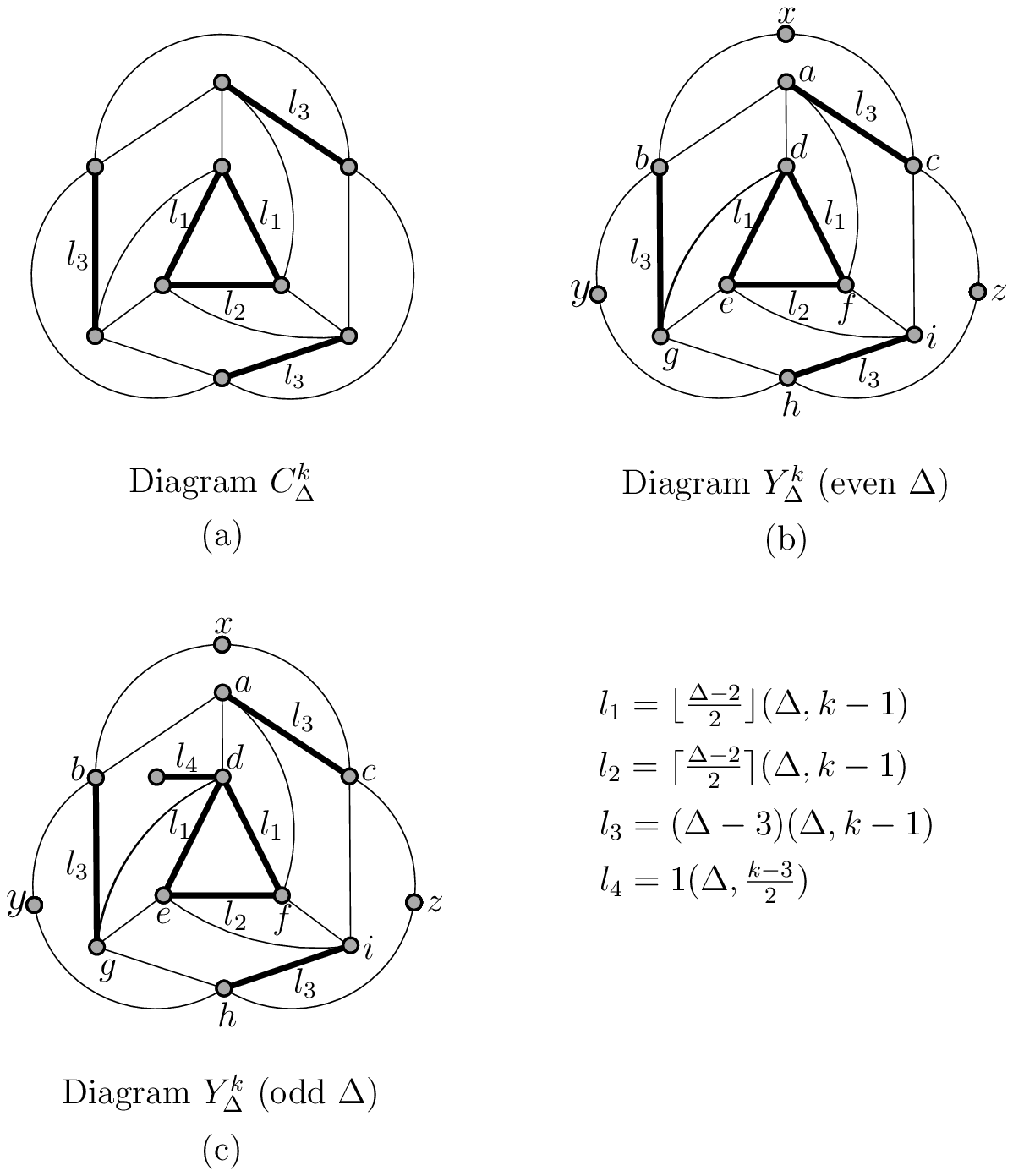}
\caption{Diagrams $C_\Delta^k$ and $Y_\Delta^k$.}
\label{fig:DiagramY}
\end{center}
\end{figure}

\begin{prop}\label{prop:G(Y)-diameter}
For odd $k\ge5$ the diameter of the graph $G(Y_{\D}^k)$ is $k$.
\end{prop}

\begin{proof}
The diagram $Y_{\Delta}^k$ does not satisfy Proposition \ref{prop:Diameter}, as the pairs of edges $(ac,bg)$, $(ac,hi)$ and $(bg,hi)$, and only those pairs, violate Condition (1). It is not difficult to verify that all other pair of edges meet the conditions of Proposition \ref{prop:Diameter}. 

The edges $ac$ and $bg$, however, are contained in the two closed walks of weight $2k+2$, namely $abgeica$ and $adgbxca$. Let $u$ be a vertex in $G(Y_{\Delta}^k)$  of a pod replacing $ac$, and $P$ the vein containing $u$. Similarly, let $u'$ be a vertex in $G(Y_{\Delta}^k)$  of a pod replacing $bg$, and $P'$ the vein containing $u'$. We observe that, since $k$ is odd, if $u$ and $u'$ are at distance $k+1$ in the closed walk $abP'geicPa$, then they cannot also be at distance $k+1$ in the closed walk $adgP'bxcPa$. This alternative to Condition (1) guarantees that the distance between any two vertices in the pods replacing $ac$ and $bg$ is at most $k$. A similar argument applies to the pairs of edges $(ac,hi)$ and $(bg,hi)$; note the symmetry in $Y_{\Delta}^k$.
\end{proof}

The number of vertices in $G(Y_{\Delta}^k)$ is
\[
|G(Y_{\Delta}^k)| =
\begin{cases}
|G(C_{\Delta}^k)|+3 & \text{if $\Delta$ is even}\\
|G(C_{\Delta}^k)|+\frac{(\D-1)^{\frac{k-3}{2}} -1}{\D-2}+3 & \text{if $\Delta$ is odd}\\
\end{cases}
\]

For odd $k\ge5$ and each $\Delta\ge6$ new largest known planar graphs arise from $G(Y_{\Delta}^k)$.

When $k\ge7$ we can do even better, by incorporating three additional pending edges to $Y_{\Delta}^k$. The resulting diagram $Z_{\Delta}^k$ is shown in Figure \ref{fig:DiagramZ}. 

\begin{figure}[!ht]
\begin{center}
\includegraphics[scale=1]{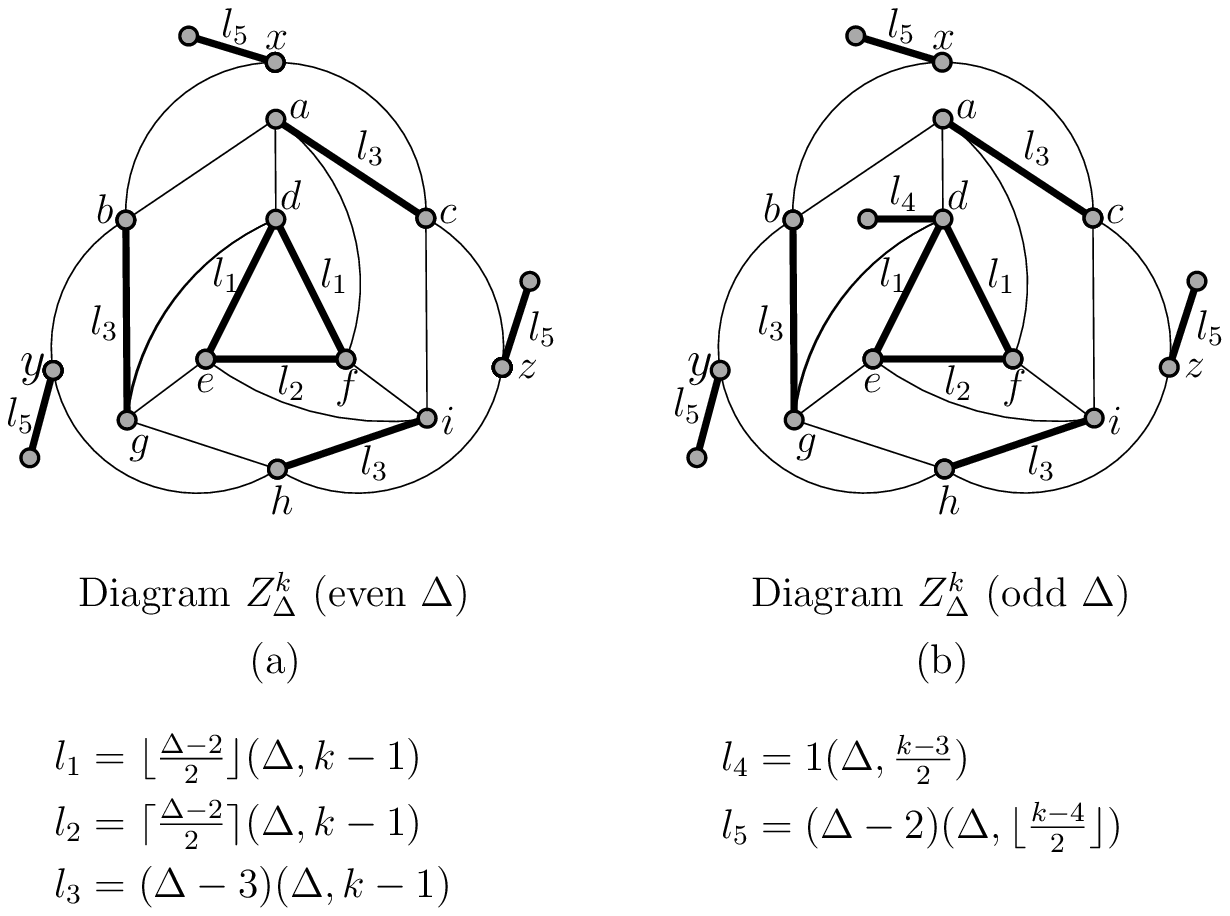}
\caption{Diagram $Z_\Delta^k$.}
\label{fig:DiagramZ}
\end{center}
\end{figure}

\begin{prop}\label{prop:G(Z)-diameter}
For odd $k\ge7$ the diameter of the graph $G(Z_{\Delta}^k)$ is $k$.
\end{prop}

\begin{proof}
By virtue of Proposition \ref{prop:G(Y)-diameter}, we only need to verify that Condition (1) of Proposition \ref{prop:Diameter}  holds for any pair of thick edges of $Z_{\Delta}^k$, in which at least one of the three additional pending edges is implicated. This fact can be verified with little effort.
\end{proof}

For the new diagram $Z_{\Delta}^k$ we have
\begin{align*}
|G(Z_{\Delta}^k)|=|G(Y_{\Delta}^k)|+3(\Delta-2)\frac{(\D-1)^{\lfloor\frac{k-4}{2}\rfloor}-1}{\Delta-2}.
\end{align*}

For odd $k\ge7$ and $\Delta\ge6$ new largest known planar graphs arise from $G(Z_{\Delta}^k)$.

The new record orders obtained from $G(Y_{\Delta}^k)$ and $G(Z_{\Delta}^k)$ have been added to the table of largest known planar graphs \cite{LPP_Planar}, and they are also displayed in Table \ref{tab:LargePlanar} of the appendix.

\section{Large graphs embedded in the torus}
\label{sec:LGTorus}
The diagram-based approach explained in the previous section can be used to produce large graphs embeddable in an arbitrary surface.

\begin{rmk}
\label{rmk:Genus}
If a diagram $\Diag_{\D}^{k}$ is embeddable in a surface $\Sigma$ then the graph $G(\Diag_{\D}^{k})$ is also embeddable in $\Sigma$.
\end{rmk}

In this section we obtain large graphs in the torus. For our constructions we will use the diagrams $P_{\D}^{k}$ (for $\D\ge3$) and $Q_{\D}^{k}$ (for $\D\ge5$ and odd diameter $k$), depicted in Figure \ref{fig:DiagramsPQ} (a) and (b), respectively. Since the Petersen graph embeds in the torus, the diagram $P_{\D}^{k}$, based on the Petersen graph, also embeds in the torus. Furthermore, $P_{\D}^{k}$ readily satisfies Proposition \ref{prop:Diameter}. Thus, we have the following. 

\begin{figure}[!ht]
\begin{center}
\includegraphics[scale=1]{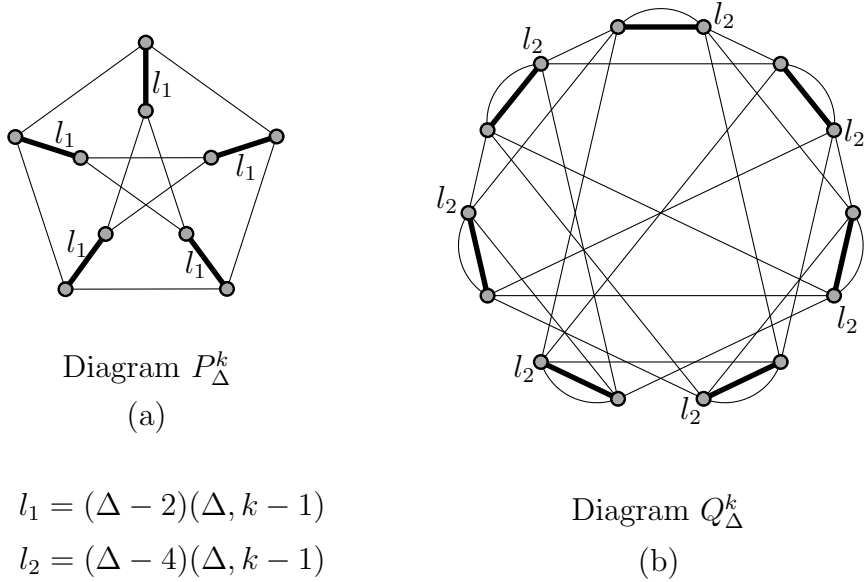}
\caption{Diagrams $P_\Delta^k$ and $Q_\Delta^k$.}
\label{fig:DiagramsPQ}
\end{center}
\end{figure}

\begin{prop}\label{prop:G(P)-diameter}
For any $k\ge3$ the diameter of the graph $G(P_{\D}^k)$ is $k$.
\end{prop}

The order of the graph $G(P_{\D}^k)$ is

\begin{displaymath}
|G(P_{\Delta}^k)| =
\begin{cases}
5\big(2(\D-1)^{\frac{k-2}{2}}-2\big)+10 & \text{if $k$ is even}\\
5\big(\Delta(\D-1)^{\frac{k-3}{2}}-2\big)+10 & \text{if $k$ is odd}\\
\end{cases} 
\end{displaymath}

An embedding of $Q_{\D}^{k}$ in the torus, based on an embedding of $K_7$, is presented in Fig.~\ref{fig:DiagramQ}. We use the drawing solution suggested in \cite[Section 2]{KocNeiSzy01}, where the torus is represented by the inner unshaded rectangle. This rectangle is surrounded by a larger, shaded rectangle, containing copies of the actual vertices and edges of the embedding. This drawing solution allows easy visualisation of the  faces and adjacency of the embedding.
 
\begin{figure}[!ht]
\begin{center}
\includegraphics[scale=1]{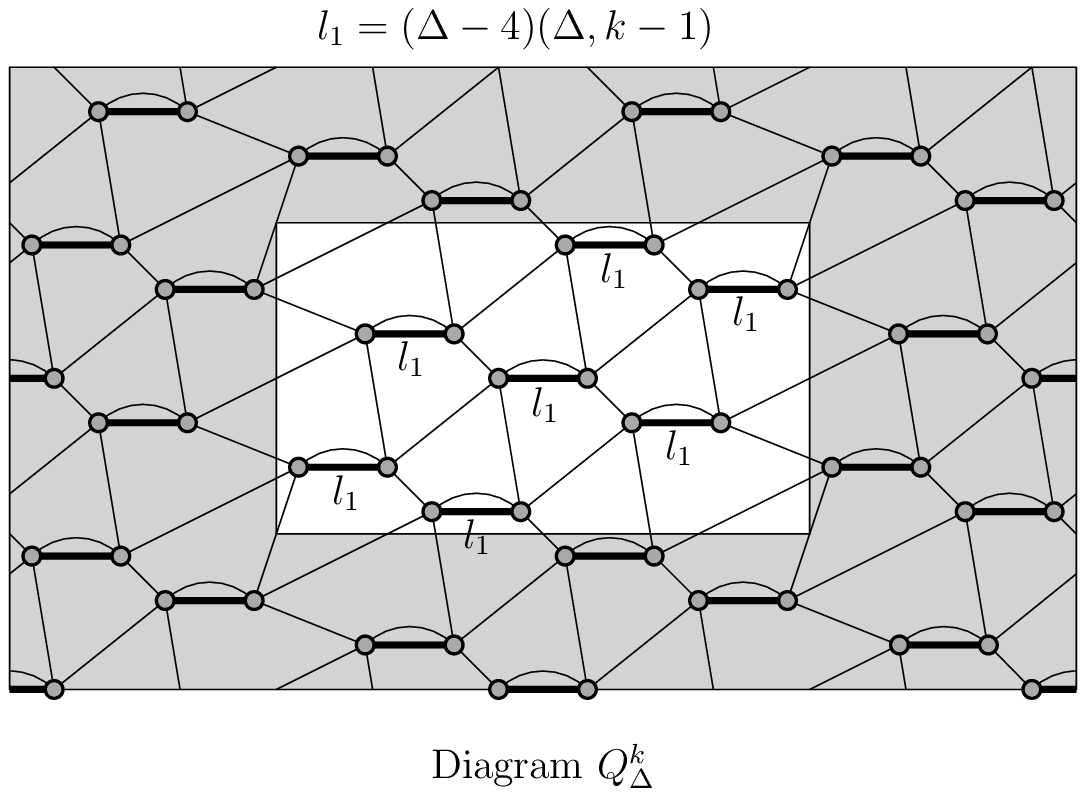}
\caption{Embedding of $Q_{\D}^{k}$ in the torus based on an embedding $K_7$.}
\label{fig:DiagramQ}
\end{center}
\end{figure}

Next we prove that $G(Q_{\D}^{k})$ has diameter at most $k$.

\begin{lem}\label{lem:BasicLemma}
Let $\Diag_{\D}^k$ be a diagram for odd $k\ge 3$. Let $e=xy$ and $e'=x'y'$ be two thick edges in $\Diag_{\D}^k$, labelled by $\alpha(\Delta,k-1)$ and $\alpha'(\Delta,k-1)$, respectively. Suppose there is a thin edge in $\Diag_{\D}^k$ joining $x$ and $x'$,  and thin edges $f=xy$ and $f=x'y'$ parallel to $e$ and $e'$, respectively. Then the distance in $G(\Diag_{\D}^k)$ between any vertex $u$ in a pod replacing $e$ and any vertex $u'$ in a pod replacing $e'$ is at most $k$. 
\end{lem}

\begin{proof}
We use a similar argument as in the proof of Proposition \ref{prop:G(Y)-diameter}. Note that, also in this case, the thick edges $e$ and $e'$ are contained in two closed walks of weight $2k+2$ (see Figure \ref{fig:BasicLemma}). Let $P$ and $P'$ be the veins in $G(\Diag_{\D}^k)$ containing $u$ and $u'$, respectively. Since $k$ is odd, if $u$ and $u'$ are at distance $k+1$ in the closed walk $xx'P'y'f'x'xPyfx$, then they cannot also be at distance $k+1$ in the closed walk $xx'f'y'P'x'xfyPx$.
\end{proof}

\begin{figure}[!ht]
\begin{center}
\includegraphics[scale=1]{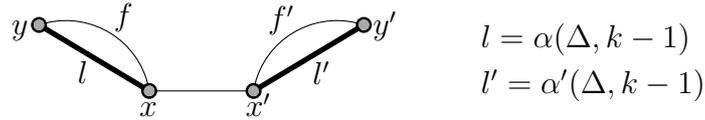}
\caption{Auxiliary figure for Lemma \ref{lem:BasicLemma}.}
\label{fig:BasicLemma}
\end{center}
\end{figure}

From Lemma \ref{lem:BasicLemma} it immediately follows 

\begin{prop}\label{prop:G(Q)-diameter}
For odd $k\ge3$ the diameter of the graph $G(Q_{\D}^k)$ is $k$.
\end{prop}

From $Q_{\Delta}^k$ we obtain

\begin{displaymath}
|G(Q_{\Delta}^k)| = 
\begin{cases}
5(\Delta-4)\frac{2(\D-1)^{\frac{k-2}{2}}-2}{\Delta-2}+14 & \text{if $k$ is even}\\
5(\Delta-4)\frac{\Delta(\D-1)^{\frac{k-3}{2}}-2}{\Delta-2}+14 & \text{if $k$ is odd}\\
\end{cases}
\end{displaymath}

The orders for toroidal graphs obtained from $P_{\Delta}^k$ and $Q_{\Delta}^k$ are displayed in Table \ref{tab:TorusTable}.

\section{Large graphs on surfaces}
\label{sec:LargeGrapSurf}

As mentioned in the introduction, Pineda-Villavicencio and Wood \cite{PVW12} constructed, for every  
surface $\Sigma$ of Euler genus $g$, every odd diameter $k\ge3$ and every maximum $\Delta\ge \sqrt{1+24g}+2$, graphs with order \[\sqrt{\frac{3}{8}g}\Delta^{\lfloor k/2\rfloor}.\]
This is the current best lower bound for $N(\Delta,k,\Sigma)$. In the following we improve this lower bound on $N(\Delta,k,\Sigma)$ by a factor of $4$, obtaining the following bound.
\[N(\Delta,k,\Sigma)\ge\begin{cases}6\Delta^{\lfloor k/2\rfloor}& \text{if $\Sigma$ is the Klein bottle}\\  \left(\frac{7}{2}+\sqrt{6g+\frac{1}{4}}\right)\Delta^{\lfloor k/2\rfloor}& \text{otherwise.}\end{cases}\] 

 Our construction modifies  a complete graph embedded in the surface $\Sigma$, so we need the Map Colouring Theorem. This theorem was jointly proved by Heawood, Ringel and Youngs; see \cite[Theorems 4.4.5 and 8.3.1]{MohTho01}.

\begin{thm}[Map Colouring Theorem]\label{theo:MapColTheo}
Let $\Sigma$ be a surface with Euler genus $g$ and let $G$ be a graph embedded in $\Sigma$. Then
\[\chi(G)\le\frac{7+\sqrt{1+24g}}{2}.\]
Furthermore, with the exception of the Klein bottle where $\chi(G)\le6$, there is a complete graph $G$ embedded in $\Sigma$ realising the equality.
\end{thm}  
 
The right-hand side of the inequality of Theorem \ref{theo:MapColTheo} is called the {\it Heawood number} of the surface $\Sigma$ and is denoted $H(\Sigma)$. Define the {\it chromatic number $\chi$ of a surface $\Sigma$} as follows:

\[\chi(\Sigma)=\begin{cases}6& \text{if $\Sigma$ is the Klein bottle}\\  H(\Sigma)& \text{otherwise.}\end{cases}\]

The main result of this section is the following.
\begin{thm}\label{thm:Main}
For every surface $\Sigma$ of Euler genus $g$, and for every $\Delta>\lceil\frac{\chi(\Sigma)-1}{2}\rceil+1$ and every odd $k\ge 3$,
\begin{align*}
N(\Delta,k,\Sigma)\ge \chi(\Sigma)\left(\Delta-1-\Big\lceil\frac{\chi(\Sigma)-1}{2}\Big\rceil\right)\frac{\D(\D-1)^{\frac{k-3}{2}}-2}{\D-2}+2\chi(\Sigma).
\end{align*}
\end{thm}

Before proving Theorem \ref{thm:Main} we recall the operation of vertex splitting. {\it Splitting a vertex $v$} consists of replacing $v$ by two adjacent vertices $v'$ and $v''$, and of replacing each edge incident with $v$ by an edge incident with either $v'$ or $v''$ leaving the other end of the edge unchanged. 

\begin{proof}[Proof of Theorem \ref{thm:Main}] We construct large graphs based on  a generalisation of the diagram $Q_{\D}^k$ in Figure \ref{fig:DiagramsPQ} (b). For a given $g$ we construct a diagram $\Q_{\D}^k$ embeddable in a surface $\Sigma$ of Euler genus $g$ such that the graph $G(\Q_{\D}^k)$ has maximum degree $\Delta$ and diameter $k$. 

To obtain $\Q_{\D}^k$ we start from the complete graph $K_{\chi(\Sigma)}$ and an embedding of $K_{\chi(\Sigma)}$ in $\Sigma$. We split every vertex $v$ in $K_{\chi(\Sigma)}$ as follows. On the surface $\Sigma$ we operate inside a neighbourhood $B_\epsilon(v)$ centred at $v$, with radius $\epsilon$ small enough so that no vertex of $K_{\chi(\Sigma)}$  other than $v$ is contained in $B_\epsilon(v)$. Take any edge of $K_{\chi(\Sigma)}$ incident with $v$ and denote it by $e_1$, then denote the other edges incident with $v$ clockwise by $e_2, e_3,\ldots,e_{\chi(\Sigma)-1}$. Split a vertex $v$ and obtain adjacent vertices $v'$ and $v''$ so that the vertex $v'$ is incident with the edges  $e_1,e_2,\ldots,e_{\lfloor\frac{\chi(\Sigma)-1}{2}\rfloor}$ and the  vertex $v''$  incident with the edges  $e_{\lfloor\frac{\chi(\Sigma)-1}{2}\rfloor+1},e_{\lfloor\frac{\chi(\Sigma)-1}{2}\rfloor+2},\ldots,e_{\chi(\Sigma)-1}$. Then a thick edge $v'v''$ labelled by $(\Delta-1-\lceil\frac{\chi(\Sigma)-1}{2}\rceil)(\Delta,k-1)$ is added; see Figure \ref{fig:DiagramsQGeneral} (b). Note that  splitting each vertex $v$ of $K_{\chi(\Sigma)}$  and the subsequent addition of one parallel thick edge do not affect the embeddability in $\Sigma$ as all these operations are carried out inside the neighbourhood $B_\epsilon(v)$.

\begin{figure}[!ht]
\begin{center}
\includegraphics[scale=1]{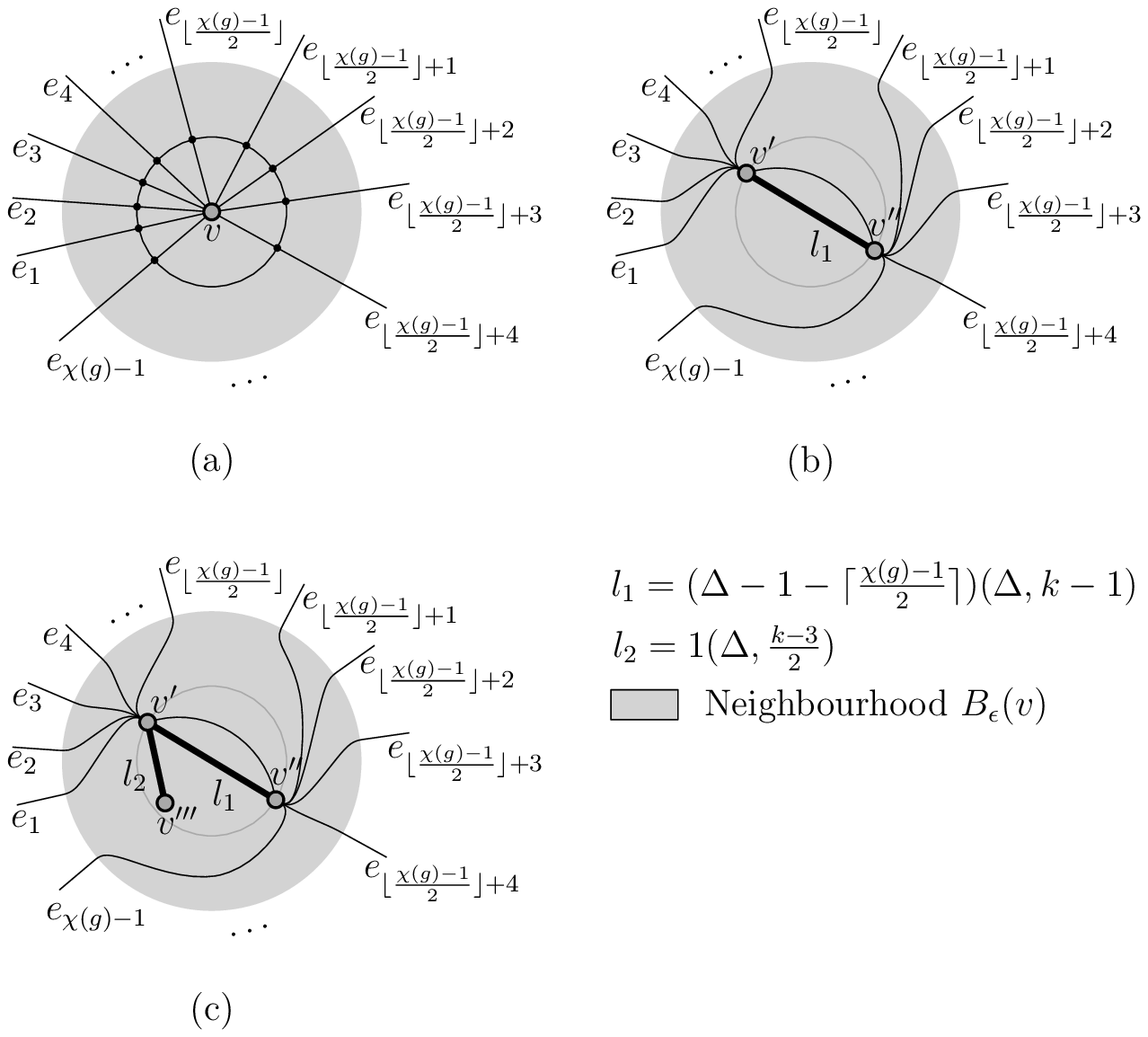}
\caption{}
\label{fig:DiagramsQGeneral}
\end{center}
\end{figure}

The resulting diagram $\Q_{\D}^k$ has maximum degree $\Delta$, and so does the graph $G(\Q_{\D}^k)$. The embeddability of $G(\Q_{\D}^k)$ follows from the embeddability of $\Q_{\D}^k$.  Note also that every thick edge in $\Q_{\D}^k$ has a parallel thin edge, and any two thick edges in $\Q_{\D}^k$ are joined by a thin edge. Thus, by Lemma \ref{lem:BasicLemma}, the diameter of $G(\Q_{\D}^k)$ is $k$. Finally we have 
\begin{displaymath}
|G(\Q_{\D}^k)|=\chi(\Sigma)\left(\Delta-1-\Big\lceil\frac{\chi(\Sigma)-1}{2}\Big\rceil\right)\frac{\D(\D-1)^{\frac{k-3}{2}}-2}{\D-2}+2\chi(\Sigma).
\end{displaymath}
\end{proof}

An example of the construction put forward in Theorem \ref{thm:Main} was already depicted in Fig.~\ref{fig:DiagramsPQ} (b); see also Fig.~\ref{fig:DiagramQ} for an embedding of such a construction in the torus. 

When $\chi(\Sigma)$ is even we can think of one improvement. Since the vertices in $\Q_{\D}^k)$ arising from $v'$ have degree $\Delta-1$, it is possible to add an extra pending edge $v'v''$ labelled by $1(\Delta,\frac{k-3}{2})$; see Figure \ref{fig:DiagramsQGeneral} (c). This would increase the order of $G(\Q_{\D}^k)$ by another $\chi(\Sigma)\frac{(\D-1)^{\frac{k-3}{2}}-1}{\D-2}$ vertices. Thus, we have the following.

\begin{cor}\label{cor:thmMain}For every surface $\Sigma$ of Euler genus $g$ and even $\chi(\Sigma)$, and for every $\Delta>\lceil\frac{\chi(\Sigma)-1}{2}\rceil+1$ and every odd $k\ge 3$, \begin{multline*}
N(\Delta,k,\Sigma)\ge\chi(\Sigma)\left(\Delta-1-\Big\lceil\frac{\chi(\Sigma)-1}{2}\Big\rceil\right)\frac{\D(\D-1)^{\frac{k-3}{2}}-2}{\D-2}\\
+\chi(\Sigma)\frac{(\D-1)^{\frac{k-3}{2}}-1}{\D-2}+2\chi(\Sigma).
\end{multline*}
\end{cor}

\section{Conclusions}

Our results and those from \cite{PVW12} imply that, for a fixed odd diameter $k$,  $N(\Delta,k,\Sigma)$ is asymptotically larger than $N(\Delta,k,\mathcal{P})$. For even diameter, however, we believe this is not the case; thus, we dare to conjecture the following.

\begin{conj}
\label{conj:FP-PV}
For each surface $\Sigma$ and each even diameter $k\ge2$, there exists a constant $\Delta_0$ such that, for maximum degree $\Delta\ge \Delta_0$, $N(\Delta,k,\Sigma)$ and $N(\Delta,k,\mathcal{P})$ are asymptotically equivalent for a fixed $k$; that is,
\[\lim_{\Delta\rightarrow \infty} \frac{N(\Delta,k,\Sigma)}{N(\Delta,k,\mathcal{P})}=1.\]
\end{conj}

Knor and \v{S}ir\'a\v{n} \cite{KS97} result for diameter 2 supports this conjecture. 

For odd $k$ we think the actual assymptotic value of $N(\Delta,k,\Sigma)$  is $(c_1+c_2\sqrt{g})\D^{\lfloor k/2\rfloor}$, where $c_1$ and $c_2$ are absolute constants. The case of $g=0$ was proved in \cite{PVW12}.

All the graphs constructed in this paper are non-regular. We could look at large regular graphs embedded in surfaces as well. This variation of the degree/diameter problem has already attracted some interest; see, for instance, \cite{Pre10}. Such direction merits further attention.

\appendix
\section{Tables of largest known planar and toroidal graphs}
\begin{table}
\def\baselinestretch{1,8}
\textwidth 5.825 in
\textheight 8.25 in
\newbox\labox\newbox\nubox\newdimen\lawi\newdimen\nuwi\newdimen\supwi
\def\degre#1{\vbox to 9mm{\vfil\hbox to9mm{{\large\sl #1}\hfil}}}
\def\BB#1#2{\setbox\labox\hbox{{\it #1}}\setbox\nubox\hbox{#2}
\lawi=\wd\labox\nuwi=\wd\nubox
\ifdim\nuwi>\lawi\supwi=\nuwi\else\supwi=\lawi\fi
\vbox to9mm{\vfil
\hbox to\supwi{\hfil\box\labox}\vfil
\hbox to\supwi{\hfil\box\nubox}\vfil}
}
\def\st{\mathord{*}}

%%%%%%%%%%%%%%%%%%%%%%%%%%%%%%%%%%%%%%%%%%%%%%%%%%%%%%%%%%%%%%%%%%%%%%
%\begin{document}

%\newpage

\begin{small}
\renewcommand{\doublerulesep}{0.1mm}
\renewcommand{\tabcolsep}{0.8mm}
\renewcommand{\arraystretch}{.1}

\noindent\begin{tabular}{||r||r|r|r|r|r|r|r|r|r||}
\hline\hline
   {\large\sl  $k$}
 & {\large\sl  2}
 & {\large\sl  3}
 & {\large\sl  4}
 & {\large\sl  5}
 & {\large\sl  6}
 & {\large\sl  7}
 & {\large\sl  8}
 & {\large\sl  9}
 & {\large\sl 10} \\
 \makebox[9mm][l]{\large $\Delta$} & & & & & & & & &\\
\hline\hline \degre{3}&
\BB{FHS}{\bf 7}&\BB{}{{\bf 12}}&\BB{}{18}&\BB{}{28}&
\BB{E}{38}&\BB{FHS}{53}&\BB{FHS}{77}&\BB{FHS}{109}&\BB{FHS}{157}\\
\hline \degre{4}&
\BB{YLD}{\bf 9}&\BB{}{16}&\BB{}{27}&\BB{FHS}{44}&
\BB{}{81}&\BB{FHS}{134}&\BB{T}{243}&\BB{FHS}{404}&\BB{FHS}{728}\\
\hline \degre{5}&
\BB{YLD}{\bf 10}&\BB{FHS}{19}&\BB{E}{39}&\BB{FHS}{73}&
\BB{T}{158}&\BB{FHS}{289}&\BB{T}{638}&\BB{FHS}{1\,153}&\BB{T}{2\,558}\\
\hline \degre{6}&
\BB{YLD}{\bf 11}&\BB{FHS}{24}&\BB{T}{55}&\BB{\underline{117}}{114}&
\BB{T}{280}&\BB{\underline{579}}{564}&\BB{T}{1\,405}&\BB{\underline{2\,889}}{2\,814}&\BB{T}{7\,030}\\
\hline \degre{7}&
\BB{YLD}{\bf 12}&\BB{FHS}{28}&\BB{T}{74}&\BB{\underline{165}}{161}&
\BB{T}{452}&\BB{\underline{984}}{959}&\BB{T}{2\,720}&\BB{\underline{5\,898}}{5\,747}&\BB{T}{16\,328}\\
\hline \degre{8}&
\BB{FHS}{\bf 13}&\BB{FHS}{33}&\BB{T}{97}&\BB{\underline{228}}{225}&\BB{T}{685}&\BB{\underline{1\,590}}{1\,569}&\BB{T}{4\,901}&
\BB{\underline{11\,124}}{10\,977}&\BB{T}{33\,613}\\
\hline \degre{9}&
\BB{FHS}{\bf 14}&\BB{FHS}{37}&\BB{T}{122}&\BB{\underline{293}}{289}&\BB{T}{986}&\BB{\underline{2\,338}}{2\,305}&
\BB{T}{7\,898}&\BB{\underline{18\,698}}{18\,433}&\BB{T}{63\,194}\\
\hline \degre{10}&
\BB{FHS}{\bf 16}&\BB{FHS}{42}&\BB{T}{151}&\BB{\underline{375}}{372}&\BB{T}{1\,366}&\BB{\underline{3\,369}}{3\,342}&
\BB{T}{12\,301}&\BB{\underline{30\,315}}{30\,072}&\BB{T}{110\,716}\\
\hline\hline
\end{tabular}
\end{small}
\begin{center}
%Table of largest known planar graphs in April 2012.
\end{center}

\caption{Table of largest known planar graphs in February  2013. Bold entries denote optimal graphs. Underlined entries correspond to the order of our largest known graphs; the old value is also recorded in the same cell. A $YLD$ acronym denotes a graph found -- or proven to be optimal -- by Yang, Lin and Dai \cite{YLD02}. A $FHS$ acronym denotes a graph found by Fellows, Hell and Seyffarth \cite{FHS98}. A $T$ acronym denotes a graph found by Tishchenko \cite{Tis2011}. An $E$ acronym denotes a graph found by Geoffrey Exoo.}
\label{tab:LargePlanar}
\end{table}

\appendix
\begin{table}
\def\baselinestretch{1,8}
\textwidth 5.825 in
\textheight 8.25 in
\newbox\labox\newbox\nubox\newdimen\lawi\newdimen\nuwi\newdimen\supwi
\def\degre#1{\vbox to 9mm{\vfil\hbox to9mm{{\large\sl #1}\hfil}}}
\def\BB#1#2{\setbox\labox\hbox{{\it #1}}\setbox\nubox\hbox{#2}
\lawi=\wd\labox\nuwi=\wd\nubox
\ifdim\nuwi>\lawi\supwi=\nuwi\else\supwi=\lawi\fi
\vbox to9mm{\vfil
\hbox to\supwi{\hfil\box\labox}\vfil
\hbox to\supwi{\hfil\box\nubox}\vfil}
}
\def\st{\mathord{*}}

%%%%%%%%%%%%%%%%%%%%%%%%%%%%%%%%%%%%%%%%%%%%%%%%%%%%%%%%%%%%%%%%%%%%%%
%\begin{document}

%\newpage

\begin{small}
\renewcommand{\doublerulesep}{0.1mm}
\renewcommand{\tabcolsep}{0.8mm}
\renewcommand{\arraystretch}{.1}

\noindent\begin{tabular}{||r||r|r|r|r|r|r|r|r|r||}
\hline\hline
   {\large\sl  $k$}
 & {\large\sl  2}
 & {\large\sl  3}
 & {\large\sl  4}
 & {\large\sl  5}
 & {\large\sl  6}
 & {\large\sl  7}
 & {\large\sl  8}
 & {\large\sl  9}
 & {\large\sl 10} \\
 \makebox[9mm][l]{\large $\Delta$} & & & & & & & & &\\

\hline\hline \degre{3}&
\BB{\underline{r}}{10}&\BB{\underline{r}}{{16}}&\BB{\underline{r}}{26}&\BB{\underline{r}}{38}&
\BB{\underline{r}}{56}&\BB{\underline{r}}{74}&\BB{\underline{r}}{92}&\BB{P}{120}&\BB{P}{160}\\

\hline \degre{4}&
\BB{\underline{r}}{13}&\BB{\underline{r}}{25}&\BB{\underline{r}}{41}&\BB{\underline{r}}{61}&
\BB{P}{90}&\BB{P}{180}&\BB{P}{270}&\BB{P}{540}&\BB{P}{810}\\

\hline \degre{5}&
\BB{\underline{r}}{16}&\BB{\underline{r}}{30}&\BB{\underline{r}}{48}&\BB{P}{100}&
\BB{P}{160}&\BB{P}{400}&\BB{P}{640}&\BB{P}{1\,600}&\BB{P}{2\,560}\\

\hline \degre{6}&
\BB{\underline{r}}{19}&\BB{\underline{r}}{37}&\BB{\underline{r}}{61}&\BB{P}{150}&
\BB{\underline{p}}{280}&\BB{P}{750}&\BB{\underline{p}}{1\,405}&\BB{P}{3\,750}&\BB{\underline{p}}{7\,030}\\

\hline \degre{7}&
\BB{\underline{p}}{12}&\BB{P}{35}&\BB{\underline{p}}{74}&\BB{P}{210}&
\BB{\underline{p}}{452}&\BB{P}{1\,260}&\BB{\underline{p}}{2\,720}&\BB{P}{7\,560}&\BB{\underline{p}}{16\,328}\\

\hline \degre{8}&
\BB{\underline{p}}{13}&\BB{P}{40}&\BB{\underline{p}}{97}&\BB{P}{280}&\BB{\underline{p}}{685}&\BB{P}{1\,960}&\BB{\underline{p}}{4\,901}&
\BB{P}{13\,720}&\BB{\underline{p}}{33\,613}\\

\hline \degre{9}&
\BB{\underline{p}}{14}&\BB{P}{45}&\BB{\underline{p}}{122}&\BB{Q}{364}&\BB{\underline{p}}{986}&\BB{Q}{2\,884}&
\BB{\underline{p}}{7\,898}&\BB{Q}{23\,044}&\BB{\underline{p}}{63\,194}\\

\hline \degre{10}&
\BB{\underline{p}}{16}&\BB{P}{50}&\BB{\underline{p}}{151}&\BB{Q}{476}&\BB{\underline{p}}{1\,366}&\BB{Q}{4\,256}&
\BB{\underline{p}}{12\,301}&\BB{Q}{38\,276}&\BB{\underline{p}}{110\,716}\\

\hline\hline
\end{tabular}
\end{small}
\begin{center}
%Table of largest known planar graphs in April 2012.
\end{center}

\caption{Table of largest known toroidal graphs in February  2013. Entries displaying $P$ and $Q$ denote graphs resulting from diagram $P_{\D}^k$ and $Q_{\D}^k$ respectively. An $r$ denotes a largest known regular toroidal graph, whereas a $p$ denotes a largest known planar graph.}
\label{tab:TorusTable}
\end{table}

\end{document}